\documentclass[reqno, letterpaper, oneside]{article}
\usepackage[utf8]{inputenc}
\usepackage{amsmath, amsthm, amsfonts,amssymb}
\usepackage{geometry}
\usepackage{hyperref}
\usepackage{mathtools, cases}

\usepackage{enumerate}

\usepackage{tikz}
\usepackage{extarrows}
\pgfdeclarelayer{nodelayer}
\pgfdeclarelayer{edgelayer}
\pgfsetlayers{edgelayer,nodelayer,main}
\tikzstyle{black}=[fill=black, draw=black, shape=circle, scale=0.3]
\tikzstyle{none}=[]

\newcommand\E{\operatorname{\mathbb E{}}}
\renewcommand\Pr{\operatorname{\mathbb P{}}}

\newcommand\Med{\operatorname{ Median{}}}

\newcommand{\eps}{\epsilon}
\newcommand{\bP}{\mathbb{P}}

\newcommand{\bE}{\mathbb{E}}

\newcommand{\bGnp}{G_{n,p}} 

\newcommand{\bfG}{G^{L}_{n,q}}
\newcommand{\ifG}{G^{S}_{n,q}}
\newcommand{\sset}{Z}	

\newtheorem{theorem}{Theorem}
\newtheorem{lemma}[theorem]{Lemma}

\newtheorem{conjecture}[theorem]{Conjecture} 
\newtheorem{remark}[theorem]{Remark}

\newcommand{\refT}[1]{Theorem~\ref{#1}}

\newcommand{\refL}[1]{Lemma~\ref{#1}}
\newcommand{\refR}[1]{Remark~\ref{#1}}
\newcommand{\refS}[1]{Section~\ref{#1}}

\newcommand{\refF}[1]{Figure~\ref{#1}}
\newcommand{\refApp}[1]{Appendix~\ref{#1}}
\newcommand{\refConj}[1]{Conjecture~\ref{#1}}

\newcommand\bigpar[1]{\bigl(#1\bigr)}
\newcommand\Bigpar[1]{\Bigl(#1\Bigr)}
\newcommand\biggpar[1]{\biggl(#1\biggr)}

\newcommand\bigsqpar[1]{\bigl[#1\bigr]}

\newcommand\Biggsqpar[1]{\Biggl[#1\Biggr]}

\newcommand\bigabs[1]{\bigl|#1\bigr|}
\newcommand\Bigabs[1]{\Bigl|#1\Bigr|}

\newcommand\ceil[1]{\lceil#1\rceil}
\newcommand\bigceil[1]{\bigl\lceil#1\bigr\rceil}

\newcommand\floor[1]{\lfloor#1\rfloor}
\newcommand\bigfloor[1]{\bigl\lfloor#1\bigr\rfloor}

\newcommand{\Gnp}{G_{n,p}}
\newcommand{\Gnq}{G_{n,q}}
\newcommand{\Gnnp}[1]{G_{{#1},p}}

\newcommand{\cD}{\mathcal{D}}
\newcommand{\cE}{\mathcal{E}}

\newcommand{\is}{\omega}
\renewcommand{\Lambda}{\mathrm{M}}

\newenvironment{romenumerate}[1][-5pt]{
\addtolength{\leftmargini}{#1}\begin{enumerate}
 }{\end{enumerate}}

\let\OLDthebibliography\thebibliography
\renewcommand\thebibliography[1]{
  \OLDthebibliography{#1}
  \setlength{\parskip}{0pt}
  \setlength{\itemsep}{0pt plus 0.3ex}
}

\title{On the concentration of the chromatic number of random graphs}
\author{Erlang Surya and Lutz Warnke%
\thanks{Department of Mathematics, University of California, San Diego, La Jolla CA~92093, USA. 
E-mail: {\tt esurya@ucsd.edu, lwarnke@ucsd.edu}. 
Supported by NSF~CAREER grant~DMS-1945481, and a Sloan Research Fellowship.}}
\date{January 3, 2022; revised November 19, 2023}
\begin{document}
	
\maketitle
\thispagestyle{empty}

\begin{abstract}
Shamir and Spencer proved in the 1980s that the chromatic number of the binomial random graph~$\Gnp$ is concentrated in an interval of length at most~$\omega\sqrt{n}$,  
and in the 1990s Alon showed that an interval of length $\omega\sqrt{n}/\log n$ suffices for constant edge-probabilities~$p\in (0,1)$. 
We prove a similar logarithmic improvement of the Shamir-Spencer concentration results for the sparse case~${p=p(n) \to 0}$, 
and uncover a surprising concentration `jump' of the chromatic number in the very dense case~${p=p(n) \to 1}$.
\end{abstract}

\section{Introduction}
What can we say about the chromatic number~$\chi(\Gnp)$ of an $n$-vertex binomial random graph~$\Gnp$?
From a combinatorial perspective, it is natural to ask about the typical value of~$\chi(\Gnp)$, i.e., upper and lower bounds that are close to each other. 
From a probabilistic perspective, it is also natural to ask about the concentration of~$\chi(\Gnp)$, i.e., how much this random variable varies. 
Among these two fundamental questions, significantly less is known about the concentration question that we shall study in this paper.

In a landmark paper from~1987, 
Shamir and Spencer~\cite{SS} proved that the chromatic number~$\chi(\Gnp)$ is typically contained in an interval of length at most~$\omega\sqrt{n}$, 
where~${\omega=\omega(n)}$ is an arbitrary function with~${\omega \to\infty}$ as $n{\to\infty}$, as usual. 
For constant edge-probabilities~${p \in (0,1)}$, 
Alon noticed in the~1990s that this concentration interval length can be slightly improved to~$\omega\sqrt{n}/\log n$,  
by adapting a coloring argument of Bollob\'{a}s~\cite{B1988}, 
see~\cite[Excercise~7.9.3]{AS} and Scott's note~\cite{Scott}. 
For uniform edge-probability~$p=1/2$, Heckel and Riordan~\cite{HR} proved in~2021 
that these old concentration bounds are in fact best possible to up poly-logarithmic~factors. 

Given the increasing knowledge about the concentration of~$\chi(\Gnp)$ in the dense case where~${p \in (0,1)}$ is constant, 
it is important to clarify our understanding of the sparse case where~${p=p(n) \to 0}$ vanishes as~${n \to \infty}$. 
For edge-probabilities of form~$p=n^{-\alpha}$ with $\alpha \in (0,1/2)$, 
Shamir and Spencer proved in their \mbox{1987 paper} that~$\chi(\Gnp)$ is typically contained in an interval of length at most~$\omega\sqrt{n}p \log n$, 
and a modern inspection of their proof reveals that length~$\omega\sqrt{n}p$ suffices. 
With Alon's improvement for constant~$p$ in mind, it is natural to wonder if 
further improvements of these sparse concentration bounds are~possible.

\enlargethispage{\baselineskip}

In this paper we sharpen the concentration of~$\chi(\Gnp)$ in the sparse case~${p=p(n) \to 0}$, 
by extending Alon's logarithmic improvement to smaller edge-probabilities: 
{\refT{SimplifiedMainSmall}} improves the Shamir-Spencer bound from~1987, 
by showing that~$\chi(\Gnp)$ is typically contained in an interval of length at most~$\omega\sqrt{n}p /\log n$. 
%
\begin{theorem}[Improved concentration bound]\label{SimplifiedMainSmall}
Let~${\omega=\omega(n)\to\infty}$ as ${n\to\infty}$ be an arbitrary function, and let~${\delta \in (0,1)}$ be a constant. 
If the edge-probability~${p=p(n)}$ satisfies ${n^{-1/2+\delta} \ll p\le 1-\delta}$, 
then there is an interval of length~$\floor{\omega\sqrt{n}p /\log n}$ 
that contains the chromatic number~$\chi(\bGnp)$ of the random graph~$\Gnp$ with high~probability, 
i.e., with probability tending to one as~${n \to \infty}$.
\end{theorem}
Our proof of \refT{SimplifiedMainSmall} refines the basic ideas of Shamir, Spencer and Alon using two greedy algorithms, 
which enable us to bypass large deviation inequalities such as Janson's inequality 
via more robust Chernoff bound based arguments. 
Note that the concentration bound of length~$\floor{\omega\sqrt{n}p /\log n}$ 
is meaningful, 
since typically~${\chi(\bGnp)=\Theta(np/\log(np))}$ holds~\mbox{\cite{B1988,L1991a,JLR}}.
Furthermore, the restriction to~${p \gg n^{-1/2+\delta}}$ is close to best possible, 
since for~${n^{-1} \ll p \ll n^{-1/2-\delta}}$ the chromatic number 
is concentrated on two different values~\mbox{\cite{L1991,AK,AN}}, 
whereas the bound from~\refT{SimplifiedMainSmall} would imply one-point concentration;
see~\refS{sec:main} for our more general chromatic number 
concentration bounds for other ranges of edge-probabilities~$p=p(n)$.
%

In this paper we also uncover a surprising concentration behavior of~$\chi(\Gnp)$ in the very dense case 
where~${p=p(n) \to 1}$ tends to one as~${n \to \infty}$:
\refT{surprise} shows that the typical length of the shortest interval containing~$\chi(\Gnp)$ 
undergoes a polynomial `jump' around edge-probability~${p=1-n^{-1+o(1)}}$; see~\refF{fig:conj}. 
\begin{theorem}[Concentration `jump' in the very dense case]\label{surprise}
Given~$\eps>0$, 
the following holds for the random graph~$\Gnp$ with edge-probability~$p=p(n)$, 
setting~$\varphi=\varphi(n,p):=n(1-p)$.%
\vspace{-0.25em}\begin{romenumerate}
	\parskip 0em  \partopsep=0pt \parsep 0em 
	\item\label{eq:surprise:i} If $n^{-o(1)} \le \varphi \ll \log n$, 
		then no interval of length~${\bigl\lfloor n^{1/2-\eps} \bigr\rfloor}$ 
		contains $\chi(\Gnp)$ with high~probability.\vspace{-0.125em}
	\item\label{eq:surprise:ii}  If $\log n\ll \varphi \le n^{o(1)}$, 
		then there is an interval of length~${\bigl\lfloor n^\eps \bigr\rfloor}$ 
		that contains $\chi(\Gnp)$ with high~probability.\vspace{-0.125em}
\end{romenumerate}
\end{theorem}
%
The concentration bounds~\ref{eq:surprise:i}--\ref{eq:surprise:ii} demonstrate 
that the typical length of the shortest interval containing~$\chi(\Gnp)$
is neither monotone nor smooth in~$p$, which both are surprising features. 
We believe that the intriguing concentration jump of~$\chi(\Gnp)$ described by \refT{surprise} 
happens infinitely many times as we vary the edge-probability~${p=1-n^{-\Omega(1)}}$; 
see~\refS{ptooneregime} and~\refF{conjfig} for more~details and further~results. 

\begin{figure}\label{conjfig}
	\centering
\begin{tikzpicture}
\begin{pgfonlayer}{nodelayer}
\node [style=none] (0) at (0, 0) {};
\node [style=none] (1) at (0, 2) {};
\node [style=none] (2) at (8, 0) {};
\node [style=none] (3) at (4, 0) {};
\node [style=none] (4) at (6, 0) {};
\node [style=none] (6) at (5.33, 2) {};
\node [style=none] (7) at (4, 2) {};
\node [style=none] (8) at (5.33, 0) {};
\node [style=none] (9) at (5.66, 0) {};
\node [style=none] (11) at (6, 2) {};
\node [style=none] (12) at (0, -0.5) {0};
\node [style=none] (13) at (4, -0.5) {1};
\node [style=none] (14) at (6, -0.5) {$\frac32$};
\node [style=none] (15) at (5.66, -0.5) {};
\node [style=none] (17) at (8, -0.5) {2};
\node [style=none] (18) at (8.75, 0) {$x$};
\node [style=none] (20) at (0, 2.5) {$y$};
\node [style=none] (22) at (5.33, -0.5) {$\frac43$};
\node [style=none] (23) at (-0.5, 2) {$\frac12$};
\node [style=none] (24) at (7, 1) {$\dots$};
\node [style=none] (25) at (6.4, -0.5) {$\frac85$};
\node [style=none] (26) at (6.4, 2) {};
\node [style=none] (28) at (6.4, 0) {};
\end{pgfonlayer}
\begin{pgfonlayer}{edgelayer}
\draw (0.center) to (7.center);
\draw[-stealth] (0.center) to (2.center);
\draw[-stealth] (0.center) to (1.center);
\draw [dotted] (7.center) to (3.center);
\draw [dotted] (6.center) to (8.center);
\draw [dotted] (11.center) to (4.center);
\draw (3.center) to (6.center);
\draw (11.center) to (8.center);
\draw (26.center) to (4.center);
\draw [dotted] (26.center) to (28.center);
\end{pgfonlayer}
\end{tikzpicture}
\caption{The exponent of the concentration interval length of very dense random graphs~$\Gnp$: 
when~${n^2(1-p)=n^{x+o(1)}}$ with~${x \in (0,2)}$, 
then \refConj{bigconj} predicts that~${n^{y+o(1)}}$ with~$y=y(x) \in [0,1/2]$ 
is the length of the shortest interval that contains~$\chi(\Gnp)$ with high probability. 
Interestingly, this proposes that the concentration interval length of $\chi(\Gnp)$ 
has infinitely many polynomial `jumps' from~${n^{1/2+o(1)}}$ to~${n^{o(1)}}$ 
as we vary the edge-probability~$p=1-n^{-2+x+o(1)}$; 
see \refT{surprise} and \refS{ptooneregime} for more~details.\label{fig:conj}}%
\end{figure}
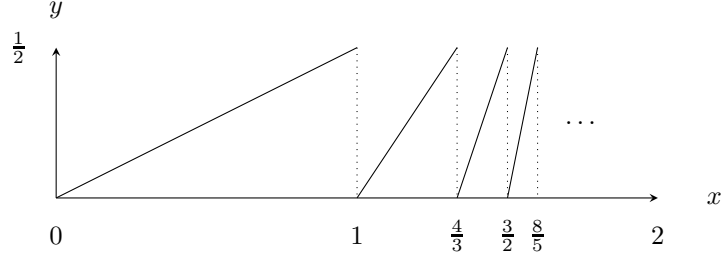


\section{Concentration bounds for~$\chi(\Gnp)$}\label{sec:main}
The following concentration result for the chromatic number~$\chi(\Gnp)$ generalizes \refT{SimplifiedMainSmall}, 
by removing the assumed lower bound on the edge-probability~$p=p(n)$. 
Interestingly, the form of the concentration bounds~\eqref{ggcase}--\eqref{llcase} changes when~$\omega\sqrt{n}p$ is around~$\log n$ 
(here we lose nothing by formally ignoring the case~${\omega\sqrt{n}p= \Theta(\log n)}$, 
since in that case we can then simply apply~\eqref{llcase} after replacing~$\omega$ with~$\sqrt{\omega}$, say).  
Note that $L=\Theta(zp/\log(zp))$ when $zp\gg (\log n)^{1+\eps}$, so that \refT{MainSmall} implies \refT{SimplifiedMainSmall} by rescaling~$\omega$. 
\begin{theorem}\label{MainSmall}
Let~$\omega=\omega(n)\to\infty$ as $n\to\infty$ be an arbitrary function, and let~$\gamma \in (0,1)$ be a constant. 
If the edge-probability~$p=p(n)$ of the random graph~$\Gnp$ satisfies~$0 < p \le \gamma$, 
then there is an interval of length at~most~$L=L(n,p,\gamma)$  
that contains the chromatic number~$\chi(\bGnp)$ with high~probability, 
with 
\begin{subequations}
\begin{numcases}{L \: := \: }
\frac{C zp}{\log(zp/\log n)} & \text{if $zp\gg \log n$,} \label{ggcase}\\
\frac{C \log n}{\log(\log n/ zp)} & \text{if $zp\ll \log n$,} \label{llcase}
\end{numcases} 	
\end{subequations}
where~$C=C(\gamma)>0$ is a constant and~$z=z(n,\omega)$ is defined as 
\begin{equation}\label{z}
z=z(n,\omega):=\omega\sqrt{n} .
\end{equation}
\end{theorem}
%


%
\refT{MainSmall} implies that~$\chi(\Gnp)$ is 
typically contained in an interval of length $O(\omega \sqrt{n}/\log n)$ for constant~${p \in (0,1)}$, 
and in an interval of length~$O(1)$ for~${p \le n^{-1/2-\delta}}$.  
These bounds match the best-known upper bounds up to constant factors~\cite{SS,AK,Scott}, 
and the best-known lower bounds up to poly-logarithmic factors~\cite{AK,HR}. 
For~${n^{-1/2} \le p \ll 1}$ we leave it as an interesting open problem whether the 
concentration bounds~\eqref{ggcase}--\eqref{llcase} for~$\chi(\Gnp)$ 
are close to best possible or~not.

The starting point for Theorems~\ref{SimplifiedMainSmall} and~\ref{MainSmall} as well as many earlier concentration proofs~\cite{SS,L1991,kriv,AK,Scott} 
is an observation about the chromatic number~$\chi(\Gnp)$ that can be traced back to Shamir and Spencer~\cite{SS}, 
which intuitively says (see~\refL{workhorse}) that with high probability the median of~$\chi(\Gnp)$ satisfies
\begin{equation}\label{eq:basic:bound}
\Bigabs{\chi(\Gnp)-\Med\bigpar{\chi(\Gnp)}} \: \le \: \max_{\sset \subseteq [n]: |\sset| \le z}\chi\bigpar{\bGnp[\sset]},
\end{equation}
where $z$ is defined as in~\eqref{z}.
The induced subgraph~$\Gnp[\sset]$ has the same distribution as~$\Gnnp{|Z|}$, 
so for suitable ranges of~$p$ we know~\cite{L1991a,JLR} that for a \emph{fixed} vertex-subset~${\sset \subseteq [n]}$ with~${|Z| \approx z}$ 
we typically have~${\chi(\bGnp[\sset])=\Theta(zp/\log(zp))}$, 
which suggests that~\refT{SimplifiedMainSmall} and~\eqref{ggcase} are 
more or less the best concentration bounds we can deduce from~\eqref{eq:basic:bound}.

Alon noted that for constant~${p \in (0,1)}$ one can bound~\eqref{eq:basic:bound} 
by adapting \mbox{Bollob\'{a}s'} analysis~\cite{B1988} 
of~$\chi(\Gnp)$ to~$\chi\bigpar{\bGnp[\sset]}$ for \emph{all} vertex-subsets~${\sset \subseteq [n]}$ with~${|Z| \le z}$. 
Here the main probabilistic ingredient are large deviation inequalities such as Janson's inequality~\mbox{\cite{Janson,RiordanWarnke2015}}: 
these allow us to show that, with high probability, all vertex-subsets of~$\Gnp$ 
with~${m \approx \sqrt{n}/(\log n)}$ vertices contain an independent set on~${k  = \Theta(\log n/p\bigr)}$ vertices. 
To color any~$\Gnp[\sset]$, we can thus iteratively remove a largest independent set from~$Z$ and assign its vertices one new color, 
until at most~$m$ uncolored vertices remain, which then each obtain a new color. 
For constant~${p \in (0,1)}$ this 
yields~${\chi\bigpar{\bGnp[\sset]} \le |\sset|/k + m}  = {O(zp/\log n)}$  
for all relevant~${\sset \subseteq [n]}$, 
which together with~\eqref{eq:basic:bound} recovers the $\omega\sqrt{n}p /\log n$ concentration bound from~\refT{SimplifiedMainSmall} up to irrelevant constant~factors.

On first sight one might think that an extra twist can extend the outlined coloring argument for~$\chi\bigpar{\bGnp[\sset]}$ to the sparse case~${p=p(n) \to 0}$ considered by~\refT{SimplifiedMainSmall}: 
indeed, a simple greedy algorithm can color the remaining~$m$ vertices with only~${O(mp)}$ colors,  
so for any relevant~${\sset \subseteq [n]}$ we should overall only~need
\begin{equation}\label{eq:optimal:2}
\chi\bigpar{\bGnp[\sset]} \: \le \: \frac{|\sset|}{k} + O\bigpar{mp} = O\biggpar{\frac{zp}{\log n}}
\end{equation}
colors, which together with~\eqref{eq:basic:bound} seemingly recovers the $\omega\sqrt{n}p /\log n$ concentration bound from~\refT{SimplifiedMainSmall} up to constant~factors. 
Unfortunately, there is another major 
 \mbox{bottleneck} we inherited from \mbox{Bollob\'{a}s'} analysis of~$\chi(\Gnp)$: 
due to union bound issues\footnote{Using large deviation inequalities inequalities the issue is that, well before~${p=n^{-1/2+\delta}}$, 
the probability that one \mbox{$m$-vertex} subset does not contain a \mbox{$k$-vertex} independent set 
is no longer small enough to take a union bond over all {$m$-vertex} subsets 
(no matter if ones uses Janson's inequality~\mbox{\cite{Janson,RiordanWarnke2015}}, Talagrand's inequality~\cite{Talagrand}, or the bounded differences inequality~\mbox{\cite{McDiarmid1989,W16}});   
this issue is also the reason why \mbox{Bollob\'{a}s'} analysis of~$\chi(\Gnp)$ breaks around~${p=n^{-1/3}}$, see~\cite[Section~3]{KM2015} and~\cite[Section~7.5]{JLR}.} 
large deviation inequalities can only guarantee \mbox{$k$-vertex} independent sets in every \mbox{$m$-vertex} subset 
as long \mbox{as~${p \ge n^{-\sigma+o(1)}}$} for some small~${\sigma>0}$. 
To extend the range of~${p=p(n)}$ we 
borrow ideas from \mbox{Grimmett} and \mbox{McDiarmid's} earlier analysis~\cite{GM1974} of~$\chi(\Gnp)$ from~1975, 
and make Chernoff bounds the main probabilistic ingredient: 
these allow us to show that, with high probability, 
all large subgraphs of~$\Gnp$ contain a vertex whose degree is small relative to the size of the subgraph, 
which enables a simple greedy algorithm to find independent sets of size $\Theta(\log n/p\bigr)$ in any \mbox{$m$-vertex} subset of~$\Gnp$ 
(this was also used by Scott~\cite{Scott} for constant~$p$).
We find it surprising that a combination of the two discussed simple greedy based refinements 
not only yields~\eqref{eq:optimal:2} and thus \refT{SimplifiedMainSmall}, 
but also establishes the more general \refT{MainSmall} by a refined analysis; 
see~\refS{sec:proofs} for the~details.


In our discussion of the concentration of~$\chi(\Gnp)$ we so far tacitly ignored the very dense case~${p=p(n) \to 1}$: 
this conveniently allowed us to work with independent sets of size~$\Theta(\log n/p)$ instead of size~$\Theta\bigpar{\log_{1/(1-p)}(n)}$, 
which in turn allowed us to write~$\Theta(p)$ instead of~${\log\bigpar{1/(1-p)}}={\sum_{k \ge 1}p^k/k}$
in the numerators appearing in \refT{SimplifiedMainSmall} as well as equations~\eqref{ggcase} and~\eqref{eq:optimal:2}. 
\refT{MainLarge} effectively says that we only need to reverse this 
simplification 
in order to extend Theorems~\ref{SimplifiedMainSmall} and~\ref{MainSmall} 
to edge-probabilities~$p=p(n) \to 1$ that tend to~one.  
\begin{theorem}\label{MainLarge}
Let~$\omega=\omega(n)\to\infty$ as $n\to\infty$ be an arbitrary function, and let~$\gamma \in (0,1)$ be a constant. 
If the edge-probability~$p=p(n)$ of the random graph~$\Gnp$ satisfies~$\gamma \le p < 1$, 
then there is an interval of length $\floor{C\omega\sqrt{n}/\log_{1/(1-p)}(n)}$ 
that contains~$\chi(\bGnp)$ with high~probability, 
where $C=C(\gamma)>0$ is a constant. 
\end{theorem}
%
This concentration bound is meaningful, since typically~${\chi(\bGnp)=\Theta\bigpar{n/\log_{1/(1-p)}(n)}}$ in the assumed range of~$p$; see~\refS{largetypical}. 
For~$p=1-n^{-\Theta(1)}$ the bound of \refT{MainLarge} effectively reduces to the $\omega\sqrt{n}$ concentration bound of Shamir and Spencer,  
which in fact is best possible for~$p=1-\Theta(n^{-1})$ due to~\refT{alonkriv} and the earlier work of Alon and Krivelevich~\cite{AK}. 
Given any integer~$r \ge 2$, we similarly believe that no further concentration improvements are possible for~$p=1-\Theta(n^{-2/r})$; see~\refF{conjfig} and~\refConj{bigconj}.
In~\refS{ptooneregime} we discuss in more detail the behavior of the chromatic number~$\chi(\Gnp)$ when~${p=1-n^{-\Omega(1)}}$.

\subsection{Proofs of \refT{MainSmall} and~\ref{MainLarge}}\label{sec:proofs} 
As discussed, the starting point for our proofs of \refT{MainSmall} and~\ref{MainLarge} is the following useful observation about the concentration of~$\chi(\Gnp)$. 
The intuition behind \refL{workhorse} is that, after removing at most~$z=\omega\sqrt{n}$ vertices from~$\Gnp$, 
we can color the remaining vertices of~$\Gnp$ with about~$\Lambda$ colors, 
where the median~${\Lambda=\Lambda(n,p)}$ of~$\chi(\Gnp)$ does not depend on~$z$ or~$\omega$ 
(which is a non-standard feature: usually a different function~${\Lambda}$ is used that explicitly depends on~$z$ or~$\omega$);  
we defer the concentration-based proof to~\refApp{sec:app}. 
\begin{lemma}[Chromatic~number: concentration around~median]\label{workhorse}
For any~${p=p(n)\in [0,1]}$ and~${z=z(n,\omega)}$ as in~\eqref{z}, 
the following holds for the random graph~$\Gnp$. 
If there is a function~${\Gamma=\Gamma(z,n,p)}$ for which the event 
\begin{equation}\label{eq:workhorse:ass}
\max_{\sset \subseteq [n]: |\sset| \le z}\chi\bigpar{\bGnp[\sset]} \: \le \: \Gamma 
\end{equation}
holds with high probability, then with high probability we also have 
\begin{equation}\label{eq:workhorse}
\bigabs{\chi(\Gnp)-\Lambda} \: \le \: \Gamma,
\end{equation}
where~${\Lambda=\Lambda(n,p)}$ is defined as the smallest integer with~${\bP(\chi(\bGnp)\le \Lambda)\ge 1/2}$. 
\end{lemma}
%
To show concentration of~$\chi(\bGnp)$, it thus suffices to show that we can color any induced subgraph~$\bGnp[\sset]$ on~$|Z| \le z$ vertices with few colors. 
To efficiently color~$\bGnp[\sset]$ we exploit local sparsity by combining standard Chernoff bounds with the following two greedy-based results, 
whose routine proofs we defer to \refApp{sec:app}. 
In brief, \refL{greedy} enables us to iteratively remove a large independent set from~$Z$ and assign its vertices one new color, until few vertices of~$Z$ remain, 
which we then color using \refL{Degeneracy}. 
Below we write~$\delta(G)$ and~$\alpha(G)$ for the minimum degree of~$G$ and the size of the largest independent set of~$G$, respectively. 
\begin{lemma}[Large independent~sets: greedy~bound]\label{greedy}%
Given a graph~$G$ and parameters~$0 < d < 1 < u$, assume that the minimum degree satisfies~$\delta(G[S]) \le d(|S|-1)$ for all vertex-subsets~$S\subseteq V(G)$ of size~$|S| \ge u$. 
Then~$\alpha(G[W]) \ge -\log_{(1-d)(1-1/u)} \bigpar{|W|/u}$ for any vertex-subset~$W \subseteq V(G)$ of size~$|W|\ge u$. 
\end{lemma}
\begin{lemma}[Chromatic~number: greedy~bound]\label{Degeneracy}%
Given a graph~$G$ and a parameter~$r \ge 0$, assume that the minimum degree satisfies~$\delta(G[S]) \le r$ for all vertex-subsets~~$S\subseteq V(G)$. 
Then~${\chi(G) \le r+1}$. 
\end{lemma}
%
In the remainder of this section we prove the concentration bounds of Theorems~\ref{MainSmall}--\ref{MainLarge} 
by combining \refL{workhorse} with variants of the above-discussed two-phase greedy coloring argument for~$\bGnp[\sset]$.

\subsubsection{Proof of \refT{MainSmall} when $zp\gg \log n$}\label{sec:zplarge}
In the case~$zp\gg \log n$ of \refT{MainSmall}, our proof strategy uses Chernoff bounds to color~$\bGnp[\sset]$ as follows. 
By iteratively applying \refL{greedy} with $u=\Theta(\log n/p)$ and $d\approx p\gg 1/u$ to extract large independent sets from~$\bGnp[\sset]$, 
in the first phase the idea is to color all but~$O(u)$ many vertices of~$Z$ using~$O\bigpar{zp/\log ({zp}/\log n)}$ colors, 
see~\eqref{eq:chiZR}~below for the details (which take into account that the largest independent sets get smaller as fewer vertices of~$Z$ remain). 
By applying~\refL{Degeneracy} with~$r=\Theta(up)$, in the second phase then idea is to then 
color the remaining~$O(u)$ uncolored vertices using only~$r+1=\Theta(\log n) \ll zp/\log ({zp}/\log n)$ additional~colors. 
%
\begin{proof}[Proof of Theorem \ref{MainSmall} when~$zp\gg \log n$]
Set $\eps:=(1-\gamma)/(1+\gamma)$, $d:=(1+\eps)p$ as well as~$u := 24 \eps^{-2} (\log n)/p$ and~$r=\ceil{2e^2up}$. 
Note that~$d < 1 \ll u$. 
With an eye on the minimum degree based Lemmas~\ref{greedy}--\ref{Degeneracy}, 
let~$\cD$ denote the event that~$\delta(\Gnp[S]) \le d(|S|-1)$ for all~$S\subseteq [n]$ with~$|S| \ge u$, 
and let~$\cE$ denote the event that~$\delta(\Gnp[S]) \le 2r-1$ for all~$S\subseteq [n]$ with $|S|\le 4 u$. 
Note that $\neg\cD$ implies existence of~$S\subseteq [n]$ with~$|S| \ge u$ such that~$\Gnp[S]$ has at least~$(1+\eps)p\tbinom{|S|}{2}$ edges.
Using standard Chernoff bounds (such as~\cite[Theorem~2.1]{JLR}) and~$\eps^2up/12=2\log n$, a routine union bound argument shows that 
\begin{equation}\label{eq:chernoff1}
\Pr(\neg \cD ) \le \sum_{u \le s \le n} \binom{n}{s} e^{-\eps^2 \binom{s}{2}p/3}
\le \sum_{u \le s \le n} \Bigpar{n e^{-\eps^2 up/12}}^s
\le \sum_{s\ge u} n^{-s} = o(1) .
\end{equation}
Exploiting~$r \ge 2e^2up > 2\log n$, we similarly see that 
\begin{equation}\label{eq:density1}
\Pr(\neg \cE ) \le \sum_{1 \le s \le 4u} \binom{n}{s}\binom{\binom{s}{2}}{sr}p^{sr}
\le \sum_{1 \le s \le 4u} \Biggsqpar{n \biggpar{\frac{4u pe}{2r}}^{r}}^s \le \sum_{s \ge 1} \Bigpar{n e^{-r}}^s\le \sum_{s \ge 1} n^{-s} = o(1) .
\end{equation}

We henceforth assume that the events~$\cD$ and~$\cE$ both hold.
Consequently, if a vertex subset~$W\subseteq [n]$ has size at least~${|W| \ge 2u}$, 
then using \refL{greedy} as well as~$1/u \ll d = (1+\eps)p \le (1+\eps)\gamma$ and~$-\log(1-d)\le d/(1-d)$ it follows~that 
the induced subgraph~$\Gnp[W]$ contains an independent set of size at least 
\begin{equation}\label{eq:largeindependent}
\alpha\bigpar{\Gnp[W]} 
\ge -\log_{(1-d)(1-1/u)}\bigpar{|W|/u}\ge \frac{\log\bigpar{|W|/u}}{-2\log\bigpar{1-d}} 
\ge \frac{\log\bigpar{|W|/u}}{A p} =: I\bigpar{|W|} 
\end{equation}
for a suitable constant~$A=A(\eps,\gamma) \in (0,\infty)$. 
In order to color any vertex subset~$\sset\subseteq [n]$ of size at most~$|\sset| \le z$, 
among the so-far uncolored vertices of~$Z$ 
we iteratively choose a largest independent set and assign its vertices one new color, 
until a set~$R \subseteq Z$ of at most~$|R| \le 4u$ uncolored vertices remains.
Applying \refL{Degeneracy} to~$\Gnp[R]$, 
in view of the event~$\cE$ and~$zp/\log n \gg 1$ 
we then color~$R$ using at~most 
\begin{equation}
\label{eq:chiR}
\chi\bigpar{\Gnp[R]} \: \le \: 2 r = \Theta\bigpar{\eps^{-2} \log n} \ll \frac{zp}{\log ({zp}/{\log n})} 
\end{equation}
many colors. 
To bound the number of colors used for~$Z\setminus R$, 
for a fixed integer~$i \ge 0$ we shall first bound the number of independent sets chosen while the number of uncolored vertices from~$Z$ is between~$z2^{-(i+1)}$ and~$z2^{-i}$, 
and then sum over all feasible integers~$i \ge 0$. 
Namely, after recalling~$u=\Theta(\eps^{-2} (\log n)/p)$ and~$zp \gg \log n$, 
it follows in view of~\eqref{eq:largeindependent} that the procedure colors~$Z \setminus R$ using at~most  
\begin{equation}
\label{eq:chiZR}
\chi\bigpar{\Gnp[Z \setminus R]} 
\: \le 
\sum_{i\ge 0: z2^{-i} \ge 4 u}\frac{z2^{-i}}{I\bigpar{z2^{-(i+1)}}}
\le \frac{A zp}{\log(z/u)} \sum_{i\ge 0}\frac{(i+2)}{2^i} \le \frac{O(A zp)}{\log ({zp}/{\log n})} 
\end{equation}
many colors, where the second inequality exploits~$z2^{-i} \ge 4 u$ to infer~${\log(z2^{-(i+1)}/u)}\ge {\log(z/u)/(i+2)}$. 
Combining~\eqref{eq:chiR}--\eqref{eq:chiZR} with \refL{workhorse} and estimates~\eqref{eq:chernoff1}--\eqref{eq:density1} then implies~\eqref{ggcase} for suitable~$C=C(\gamma)>0$.  	
\end{proof}

\subsubsection{Proof of \refT{MainSmall} when~$zp\ll \log n$}\label{sec:zpsmall}
In the remaining case~$zp\ll \log n$ of \refT{MainSmall}, we can directly bound~$\chi(\bGnp[\sset])$ using~\refL{Degeneracy}. 
%
\begin{proof}[Proof of \refT{MainSmall} when~$zp\ll \log n$]
Set~$r := \bigceil{{4\log n}/{\log\bigl(\frac{\log n}{zp}\bigr)}}$ and~$C:=16$, say. 
Let~$\cE$ denote the event that~$\delta(\Gnp[S]) \le 2r-1$ for all~$S\subseteq [n]$ with~$|S| \le z$. 
Noting that~$zp\ll \log n$ implies~$zp/r \ll \sqrt{zp/\log n}$, similar to~\eqref{eq:density1} it follows via a standard union bound argument that 
\begin{align*}
\Pr(\neg \cE ) \le \sum_{1 \le s \le z} \binom{n}{s}\binom{\binom{s}{2}}{sr}p^{sr}
\le \sum_{1 \le s \le z} \Biggsqpar{n \biggpar{ \frac{zpe}{2r}}^{r}}^s 
\le \sum_{s\ge 1} \Biggsqpar{n \biggpar{ \frac{zp}{\log n}}^{r/2}}^s
\le \sum_{s\ge 1} n^{-s} = o(1) .
\end{align*}
We henceforth assume that the event~$\cE$ holds.
Applying~\refL{Degeneracy}, we then color any~$\sset\subseteq [n]$ with~$|\sset| \le z$ 
using at most~$\chi(\bGnp[\sset]) \le 2r$ many colors, 
which together with \refL{workhorse} implies~\eqref{llcase} for~$r \ge 2$ 
(in which case the ceiling in the definition of~$r$ causes no major rounding~issues).

In the remaining case~$r=1$ it is easy to see that $p \ll n^{-4}$ holds (with room to spare), so that $\Gnp$ has with high probability no edges.
Hence~$\chi(\Gnp)$ is concentrated on one value, which trivially establishes~\eqref{llcase}.
\end{proof}

\subsubsection{Proof of \refT{MainLarge}}\label{sec:MainLarge}
The proof strategy for \refT{MainLarge} is similar but simpler to \refS{sec:zplarge}.
Namely, by iteratively applying~\refL{greedy} to extract independent sets of size~$\Theta\bigpar{\log_{1/(1-p)}(n)}$ from~$\bGnp[\sset]$, 
we first color all but at most~$n^{1/3}$ many vertices of~$Z$ using~$O\bigpar{z/\log_{1/(1-p)}(n)}$ colors, 
and then trivially color the remaining uncolored vertices using at most~$n^{1/3} \ll z/\log_{1/(1-p)}(n)$ additional colors (by giving each vertex a new~color).  
\begin{proof}[Proof of Theorem \ref{MainLarge}]
Set~$q:=1-p$, $d:=1-q/2$ as well as~$u:=96(\log n)/q$ and~$m:=n^{1/3}$. 
If~$q \le n^{-1/9}$ then \refL{workhorse} implies the claimed bound in \refT{MainLarge} by noting that~$\chi(\bGnp[\sset]) \le |\sset| \le z \le 9 z\log(1/q)/\log n$, 
so we henceforth assume that~$q \ge n^{-1/9}$.
Let~$\cD$ denote the event that~$\delta(\Gnp[S]) \le d(|S|-1)$ for all~$S\subseteq [n]$ with~$|S| \ge u$.
Note that $\neg\cD$ implies existence of~$S\subseteq [n]$ with~$|S| \ge u$ such that~$\Gnp[S]$ has at most~$q/2 \cdot \tbinom{|S|}{2}$ non-edges.
Similarly to~\eqref{eq:chernoff1}, 
using~$1-p=q$ and~$uq/48 = 2 \log n$ it routinely follows that
\begin{equation}\label{density2}
\bP(\neg \cD)\le \sum_{u\le s\le n} \binom{n}{s}e^{-\binom{s}{2}q/12}\le \sum_{u\le s\le n}\Bigpar{n e^{-uq/48}}^s\le \sum_{s \ge u} n^{-s} =o(1).
\end{equation}

We henceforth assume that the event~$\cD$ holds.
Consequently, if a vertex subset~$W\subseteq [n]$ has size at least~${|W| \ge m}$, 
then using \refL{greedy} together with $m/u \gg n^{1/6}$ as well as~$1/u \ll p = 1-q$ and~$q \le 1-\gamma$ it follows that
the induced subgraph~$\Gnp[W]$ contains an independent set of size at least 
\begin{equation*}
\alpha\bigpar{\Gnp[W]} 
\ge -\log_{[q/2 \cdot (1-1/u)]}\bigpar{|W|/u}
\ge \frac{\log(m/u)}{\log(2/q^2)} 
\ge \frac{\log n}{A \log(1/q)}=:k 
\end{equation*}
for a suitable constant~$A=A(\gamma) \in (0,\infty)$. 
%
In order to color any vertex subset~$\sset\subseteq [n]$ of size at most~$|\sset| \le z$, 
among the so-far uncolored vertices of~$Z$ 
we iteratively choose an independent set of size at least~$k$ and assign its vertices one new color, 
until at most~$m$ uncolored vertices remain, 
and then use one new color for each remaining vertex. 
In view of~$|Z| \le z$ and~$m \ll z/k$ it follows that this procedure colors~$Z$ using at~most  
\begin{equation*}
\chi\bigpar{\Gnp[Z]} \: \le \: \frac{|Z|}{k}+m 
\le \frac{2z}{k} \le \frac{2A z}{\log_{1/q}(n)} 
\end{equation*}
many colors, which together with \refL{workhorse} and estimate~\eqref{density2} completes our proof with~$C:=\max\{4A,18\}$. 
\end{proof}

\section{Concentration of $\chi(G_{n,p})$: the very dense case}\label{ptooneregime} 
We conclude by discussing the behavior of the chromatic number~$\chi(\Gnp)$ in the very dense case~${1-p=n^{-\Omega(1)}}$. 
Here $\chi(\Gnp)$ is closely linked to the size~$\alpha=\alpha(\Gnp)$ of the largest independent set of~$\Gnp$.
Namely, inspired by~\cite{HR}, it seems plausible that a near-optimal coloring can be obtained by first picking as many vertex-disjoint independent sets of size~$\alpha$ as possible, 
and then covering (almost all of) 
the remaining vertices with independent sets of size~$\alpha-1$. 
More concretely, for~$n^{-2/r} \ll 1-p \ll n^{-2/(r+1)}$ we expect to have~about
\begin{equation}\label{eq:mur1}
\mu_{r+1}=\mu_{r+1}(n,p):=\binom{n}{r+1}(1-p)^{\binom{r+1}{2}}
\end{equation}
many vertices in independent sets of largest size~$\alpha=\alpha(\Gnp)=r+1$ (which in fact are mostly vertex-disjoint). 
Since~$\mu_{r+1} = o(n)$, the near-optimal coloring heuristic then suggests that~$\chi(\Gnp) \approx o(n)+n/r \approx n/r$, as we shall make rigorous in~\refS{largetypical}.
Furthermore, since in our coloring heuristic we pick almost all independent sets of size~$r+1$, whose number is well-known to fluctuate by about~$\sqrt{\mu_{r+1}}$ in~$\Gnp$ (see~\cite{ruc}), 
it then becomes plausible that~$\chi(\Gnp)$ should also vary by this amount, 
as formalized by the following~conjecture. 
\begin{conjecture}\label{bigconj}
If the edge-probability~$p=p(n)$ satisfies~$(\log n)^{1/\binom{r}{2}} n^{-2/r} \ll 1-p \le (1+o(1)) n^{-2/(r+1)}$ for some integer~$r \ge 1$, 
then the following holds, for any~$\eps>0$ and any function~$\omega=\omega(n)\to\infty$ as $n\to\infty$. 
\vspace{-0.125em}\begin{romenumerate}
\parskip 0em  \partopsep=0pt \parsep 0em 
\item\label{conj:upper}
There is an interval of length~${\bigl\lfloor \omega \sqrt{\mu_{r+1}}\bigr\rfloor}$ that contains $\chi(\Gnp)$ with high~probability.
\item\label{conj:lower}
No interval of length~${\bigl\lfloor c \sqrt{\mu_{r+1}}\bigr\rfloor}$ 
contains $\chi(\Gnp)$ with probability at least~$\eps+o(1)$, where~$c=c(\eps,r)>0$.  
\end{romenumerate}
\end{conjecture}
Note that~$\sqrt{\mu_{r+1}}$ increases from~$\Theta(1)$ to~$\Theta(\sqrt{n})$ as~$1-p$ increases from~$n^{-2/r}$ to~$n^{-2/(r+1)}$, 
so \refConj{bigconj} predicts that the concentration interval length of~$\chi(\Gnp)$ is neither monotone nor smooth in~$p$. 
Furthermore, \refConj{bigconj}\ref{conj:upper}  proposes a refinement of the well-known $\omega \sqrt{n}$ concentration bound of Shamir and Spencer~\cite{SS}. 
More interestingly, by varying~$r \ge 1$, \refConj{bigconj} also predicts that the concentration interval length changes infinitely many times from~$n^{1/2+o(1)}$ to $(\log n)^{\Theta(1)}$ as~$1-p$ decreases from~$n^{-o(1)}$ to~$n^{-2+o(1)}$; see also \refF{fig:conj}. 
This intriguing behavior differs conceptually from very recent predictions~\cite{HR} for constant~$p \in (0,1)$.

The remainder of this section is organized as follows. 
In \refS{small} we prove \refConj{bigconj} for~$r=1$, and in \refS{large} we establish \refT{surprise}, 
i.e., show that the concentration interval length of $\chi(\Gnp)$ undergoes a polynomial `jump' around~$1-p=n^{-1+o(1)}$. 
In \refS{largetypical} we then prove that, under the assumptions of \refConj{bigconj}, the chromatic number is typically~$\chi(\Gnp) =(1+o(1)) n/r$, as predicted~above.

\subsection{Concentration result: proof of \refConj{bigconj} for~$r=1$}\label{small}
The following result verifies the case~$r=1$ of \refConj{bigconj}:  
it states that~$\chi(\bGnp)$  is concentrated on an interval of length about~$\sqrt{n} \cdot \sqrt{n(1-p)}$ when~$n^{-2} \ll 1-p=O(1/n)$.  
In this range of~$p$, \refT{alonkriv}\ref{thm:alonkriv:conc} refines the well-known $\omega \sqrt{n}$ concentration bound of Shamir and Spencer~\cite{SS}. 
\refT{alonkriv}\ref{thm:alonkriv:nonconc} extends observations of Alon and Krivelevich~\cite{AK} and Bollob\'{a}s~\cite{B2004} for~$p=1-1/(10n)$ and~$n^{-2} \ll 1-p \ll n^{-3/2}$, respectively. 
\begin{theorem}\label{alonkriv}
If~$p=p(n)$ satisfies~$n^{-2}\ll 1-p \le (D+o(1))/n$ for some constant~${D \in (0,\infty)}$, 
then for~${q:=1-p}$ the following holds, for any~$\eps>0$ and any function~$\omega=\omega(n)\to\infty$ as $n\to\infty$. 
\vspace{-0.125em}\begin{romenumerate}
\parskip 0em  \partopsep=0pt \parsep 0em 
\item\label{thm:alonkriv:conc} 
There is an interval of length~${\bigl\lfloor \omega n\sqrt{q} \bigr\rfloor}$ that contains $\chi(\Gnp)$ with high~probability. 
\item\label{thm:alonkriv:nonconc} 
No interval of length~${\bigl\lfloor cn\sqrt{q}\bigr\rfloor}$ 
contains $\chi(\Gnp)$ with probability at least~$\eps+o(1)$, where~$c=c(\eps,D)>0$.  
\end{romenumerate}
\end{theorem}
\begin{proof}%
The main idea is that the study of study of~$\chi(\bGnp)$ effectively reduces to the study of the maximum matching in the complement.
Writing~$M$ for the maximum size (counting number of edges) of a matching in the complement of~$\Gnp$, we have~$\chi(\Gnp) \le M + (n-2M) = n-M$.
Furthermore, for any proper coloring of~$\Gnp$ with color classes~$(I_s)_{1 \le s \le \chi(\Gnp)}$ we have~$n=\chi(\Gnp) + \sum_{s}(|I_s|-1)$.
Defining~$X_i$ as the number of independent sets of size~$i$ in $\Gnp$, for~$Y := \sum_{i \ge 3} i X_i$ it then readily follows~that 
\begin{equation}\label{tightineq}
n-M-Y\le \chi(\Gnp)\le n-M.
\end{equation}
Since~$\bE X_{i+1} /\bE X_{i} = \mu_{i+1}/\mu_i \le nq^i=o(q)$ for $i\ge 3$, it routinely follows that~$\E Y = (1+o(1)) 3\mu_3 = \Theta(n^3q^3) = o(n\sqrt{q})$.  
Applying Markov's inequality to~$Y$, from~\eqref{tightineq} it follows that with high probability 
\begin{equation}\label{eq:alonkriv}
\chi(\Gnp)= n-M-o(n\sqrt{q}). 
\end{equation}

Gearing up towards applying a combinatorial version of Talagrand's inequality, 
note that $M\ge s$ can be certified by~$s$ non-edges of~$\Gnp$, 
and that adding or removing an edge from $\Gnp$ changes~$M$ by at most one. 
Since~$\bE M\le \binom{n}{2}q$, a standard application of Talagrand's inequality (such as~\cite[Theorem 2]{Talagrand}) thus yields
\begin{equation*}
\bP\bigpar{|M-\bE M| \ge \omega n\sqrt{q }/3} \: \le \: 2 \cdot \exp\biggpar{-\Theta\biggpar{\frac{\omega^2n^2q }{\bE M +\omega n\sqrt{q}}}}=o(1),
\end{equation*}
which together with~\eqref{eq:alonkriv} 
completes the proof of case~\ref{thm:alonkriv:conc}.

Finally, the remaining case~\ref{thm:alonkriv:nonconc} follows immediately from~\eqref{eq:alonkriv} and \refL{isoledges} below 
(since the complement of~$\Gnp$ has the same distribution as~$\Gnq$, where~$qn \le D+o(1)$ and $n\sqrt{q} \gg 1$ hold by assumption).
\end{proof}
\begin{lemma}\label{isoledges}
Let~$q=q(n) \in [0,1]$ satisfy~$n\sqrt{q}e^{-qn} \gg 1$, and define~$M$ as the maximum size (counting number of edges) of a matching in~$\Gnq$.  
Then for any~$\eps>0$ there is a constant~$d=d(\eps)>0$ such that for any interval~$I$ of length~${\bigl\lfloor dn\sqrt{q} e^{-qn} \bigr\rfloor}$ we have~$\bP(M\in I)\le \eps+o(1)$. 
\end{lemma}
\begin{proof}[Proof-Outline of \refL{isoledges}]
The heuristic idea is that fluctuations in the number of isolated edges is a source of fluctuations of~$M$. 
Defining~${\bfG \subseteq \Gnq}$ as the induced subgraph of~$\Gnq$ consisting of all connected components of size at least three, we set $Y:=n-|V(\bfG)|$.
We also define~$X_1$ as the number of isolated edges in~$\Gnq$, and define~$X_2$ as the maximum size of a matching in~$\bfG$. 
The main idea is then to show that, after conditioning on a fixed graph~${\bfG=G}$ with~${Y \approx \E Y}$, the number~$X_1$ of isolated edges fluctuates by ${\Theta(\sqrt{\E X_1}) = \Theta(n\sqrt{q} e^{-qn})}$. 
Since~$X_2$ is determined by the graph~$G$ we conditioned on, this then allows us to show that~$M={X_1+X_2}$ fluctuates by the same order of magnitude. 
See \refApp{apx:isoledges} for the full technical details of the proof of \refL{isoledges} (which are rather tangential to the other arguments~here).  
\end{proof}
\begin{remark}[Extension of \refT{alonkriv}]\label{extend}
After refining the deviation from~$cn\sqrt{q}$ to~$cn\sqrt{q}e^{-nq}$ for suitable~${c=c(\eps)>0}$, 
the above proof of \refT{alonkriv}~\ref{thm:alonkriv:nonconc} carries over under the weaker assumption~$n\sqrt{q}e^{-nq}\gg \max\{1,(nq)^3\}$, 
which in particular holds for when $n^{-1} \ll nq \ll \log n$, say. 
\end{remark}

\subsection{Polynomial concentration jump: proof of \refT{surprise}}\label{large}
%
%
The following simple lemma proves a weak version of \refConj{bigconj}\ref{conj:upper} for~$r \ge 2$, 
with deviation~$\omega \sqrt{\mu_{r+1}}$ replaced by~$\omega \mu_{r+1}$. 
For our purposes, the crux is that \refL{polylogconc} establishes concentration of~$\chi(\bGnp)$ on~$n^{o(1)}$ values 
near the transition points~$1-p=n^{-2/r+o(1)}$.
By combining the case~$r=2$ of \refL{polylogconc} with \refR{extend} we thus immediately establish~\refT{surprise}, 
which essentially says that the typical length of the shortest interval containing~$\chi(\Gnp)$ 
undergoes a steep transition from~$\Omega(n^{1/2-\eps})$ down to~$O(n^{\eps})$ around edge-probability~${p=1-n^{-1+o(1)}}$, 
as illustrated by~\refF{fig:conj} and predicted by~\refConj{bigconj}. 
This polynomial concentration `jump' is intriguing, and we remark that a related phenomenon also occurs in~\cite{MMP,LMW}.
\begin{lemma}\label{polylogconc}
Let~$\omega=\omega(n)\to\infty$ as $n\to\infty$ be an arbitrary function, and let~$r \ge 2$ be an integer.
If~$p=p(n)$ satisfies~$(\log n)^{1/\binom{r}{2}} n^{-2/r}\ll 1-p \ll (\log n) n^{-2/(r+1)}$, 
then there is an interval of length~${\bigl\lfloor \omega \mu_{r+1} \bigr\rfloor}$ that contains $\chi(\Gnp)$ with high~probability. 
Furthermore, for~$1-p=n^{-2/r}x$ we have~$\omega \mu_{r+1} \le \omega x^{\binom{r+1}{2}}$.  
\end{lemma}
\begin{proof}
Using the assumed lower bound on~$q:=1-p$, a celebrated result of Johansson, Kahn and Vu~{\cite[Theorem~2.1]{JKV}}
implies that the complement of~$\bGnp$ with high probability contains$\floor{n/r}$ vertex-disjoint cliques of size~$r$, 
meaning that we can color all but~$n-r\floor{n/r} \le r$ vertices of~$\Gnp$ using at most~$n/r$ colors. 
By coloring the remaining uncolored vertices 
with distinct new colors, 
it follows that with high~probability 
\begin{equation}\label{eq:dense:upper}
\chi(\Gnp) \le n/r+r.
\end{equation}

To obtain a lower bound on~$\chi(\Gnp)$ we shall use a variant of the argument leading to~\eqref{tightineq}, again writing~$X_i$ for the number of independent sets of size~$i$ in~$\Gnp$. 
Indeed, for any proper coloring of~$\Gnp$ with color classes~$(I_s)_{1 \le s \le \chi(\Gnp)}$ we have~$n=\sum_{s}|I_s| \le Y + r \cdot \chi(\Gnp)$ for~$Y:=\sum_{i > r}iX_i$, so that 
\begin{equation}\label{eq:dense:lower}
\chi(\Gnp) \ge (n-Y)/r .
\end{equation}
Using the assumed upper bound on~$q=1-p$ we infer that~$\bE X_{i+1} /\bE X_{i} \le nq^i=o(q)$ for $i\ge r+1$ (with room to spare), and so it routinely follows that~$\E Y = \Theta(\mu_{r+1})$. 
Finally, in view of~\eqref{eq:dense:upper}--\eqref{eq:dense:lower}, 
now an application of Markov's inequality to~$Y$ readily completes the proof. 
\end{proof}
The deviation~$\omega \mu_{r+1}$ in \refL{polylogconc} can be further reduced using a two-round exposure argument (in some range of~$p$), 
but a proof of the conjectured deviation~$\omega \sqrt{\mu_{r+1}}$ seems to require additional~ideas. 

\subsection{The typical value}\label{largetypical}
Finally, an approximate coloring argument similar to \refS{large} also allows us to determine the with high probability asymptotics of the chromatic number~$\chi(\bGnp)$ 
for most edge-probabilities of form~$1-p=n^{-\Omega(1)}$, 
settling a recent conjecture of Isaev and Kang, see~\cite[Conjecture~1.2]{IK}. 
\begin{theorem}\label{isaevkang}
If the edge-probability~${p=p(n)}$ satisfies~${n^{-2/r} \ll 1-p \ll n^{-2/(r+1)}}$ for some fixed integer~${r \ge 1}$, 
then with high probability~${\chi(\bGnp)=(1+o(1))n/r}$.
\end{theorem}
\begin{proof}
For the lower bound we exploit that the argument leading to~\eqref{eq:dense:lower} 
again gives~$\chi(\bGnp)\ge (n-Y)/r$ and~$\E Y  = \Theta\Bigpar{n^{r+1}(1-p)^{\binom{r+1}{2}}}$. 
Since~$1-p \ll n^{-2/(r+1)}$ implies~$\E Y = o(n)$, now an 
application of Markov's inequality to~$Y$ readily shows that with high probability~$\chi(\bGnp)\ge (1-o(1))n/r$.

We also include the upper bound argument from~\cite{IK} for completeness. 
Namely, using~$1-p \gg n^{-2/r}$, an old 
result of Ruci{\'n}ski~\cite[Theorem~4]{rucpm} implies that 
the complement of~$\bGnp$ with high probability contains~${(1-o(1))n/r}$ vertex-disjoint cliques of size~$r$, 
meaning that we can color all but~$o(n)$ vertices of~$\Gnp$ using at most~$n/r$ colors. 
By coloring the remaining uncolored vertices 
with distinct new colors, 
it follows that with high probability~$\chi(\bGnp) \le n/r + o(n) \le (1+o(1))n/r$, 
completing the proof. 
\end{proof}
Note that the estimate~$\chi(\bGnp)=(1+o(1))n/r$ from \refT{isaevkang} can also be rewritten as 
\begin{equation}\label{eq:chi:Gnp}
\chi(\bGnp) = \frac{(1+o(1))n}{\bigfloor{2\log_{1/(1-p)}(np)}}.
\end{equation}
Interestingly, the same expression~\eqref{eq:chi:Gnp} also gives the with high probability asymptotics of~$\chi(\bGnp)$ 
when the edge-probability~${p=p(n)}$ satisfies~$1/n \ll p \le 1-n^{-o(1)}$, see~\cite{B1988,L1991a,IK}. 
We leave it as an intriguing open problem to determine the asymptotics of~$\chi(\bGnp)$ for edge-probabilities of the form~$p=1-\Theta(n^{-2/r})$,
in which case the behavior of related graph parameters also remains open, see~\cite[Section~4.2]{GPW}.

\bigskip{\noindent\bf Acknowledgements.} 
We thank Annika Heckel for valuable discussions about~\refConj{bigconj}. 
We are also grateful to the referees and Christian Houdr{\'e} for helpful suggestions concerning the presentation.

\footnotesize
\bibliographystyle{plain}

\normalsize

\begin{appendix}

\section{Appendix: proofs of Lemmas~\ref{workhorse}--\ref{Degeneracy}}\label{sec:app}
%
%
\begin{proof}[Proof of \refL{workhorse}]	
Define ${\lambda=\lambda(n,p,\is)}$ as the smallest integer with ${\bP(\chi(\bGnp)\le \lambda)\ge 1/\is}$. 
Let~$Y$ denote the minimum size of a vertex subset ${\sset\subseteq [n]}$ with ${\chi\bigpar{\bGnp\bigsqpar{[n] \setminus \sset}}\le \lambda}$. 
By definition of~$\lambda$ we~have 
\begin{equation}\label{eq:wh:anchor}
\bP(\chi(\Gnp)< \lambda) \: \le \: 1/\is \: \le \: \bP(Y=0).
\end{equation}
Using the vertex-exposure approach to the bounded difference inequality,
we write~$Y=Y(X_1,\dots, X_n)$ where the independent auxiliary variables $X_i:=\{(i,j)\in E(\Gnp) \: : \: i<j\}$ contain the edges of~$\Gnp$ between vertex~$i$ and vertices~$\{i+1, \ldots, n\}$. 
Note that changing a single~$X_i$ can change~$Y$ by at most one. 
In view of~\eqref{eq:wh:anchor}, now a routine application of the bounded differences inequality (such as~\cite[Corollary~2.7]{JLR}) to~$Y$ yields 
\[ 1/\is \le \bP(Y=0) = \bP(Y \le \E Y - \E Y)\le \exp\biggpar{\frac{-(\bE Y)^2}{2n}}, \]
which implies that~$\bE Y\le \sqrt{2n\ln \is}$.  
Since~$z=\is \sqrt{n} \gg \sqrt{2n\ln \is}$, 
by again applying the bounded differences inequality to~$Y$ it then follows~that 
\[\bP(Y\ge z )\le \bP\bigpar{Y\ge \bE Y+\sqrt{2n\ln \is}}\le 
1/\is ,\]
which together with estimate~\eqref{eq:wh:anchor} and the `with high probability' event~\eqref{eq:workhorse:ass} implies that 
\begin{equation}\label{eq:first}
\bP\bigpar{\lambda \le \chi(\Gnp) \le \lambda + \Gamma } \ge 1- 2/\is-\Pr\bigpar{\text{event~\eqref{eq:workhorse:ass} fails}} = 1-o(1).
\end{equation}
Combining the definition of the median~$\Lambda=\Lambda(n,p)$ with~\eqref{eq:wh:anchor} and~\eqref{eq:first}, it deterministically follows that
\[  \lambda \: \le \: \Lambda \le \lambda+\Gamma \]
for all sufficiently large~$n$, which together with~\eqref{eq:first} then completes the proof of \refL{workhorse}. 
\end{proof}
\begin{proof}[Proof of \refL{greedy}]
Given a vertex subset $W\subseteq V(G)$ of size~$|W| \ge u$, we construct an independent set greedily: 
set $W_0=W$ and, for $i\ge 1$, pick $w_i\in W_{i-1}$ with minimal degree in $G[W_{i-1}]$ and set 
\[W_i=\bigl\{v\in W_{i-1} \ : \  v \text{ not adjacent to } w_i \bigr\}.\] 
We terminate as soon as~$W_j$ is empty, in which case we obtain an independent set~$\{w_1,\dots, w_j\} \subseteq W$. 
If~$|W_{i-1}|\ge u$ holds, then we know that~$w_i$ has degree at most $d(|W_i|-1)$ in $G[W_i]$, implying that 
\begin{equation*}
|W_i|
\ge (1-d)(|W_{i-1}|-1)
\ge  (1-d)(1-1/u)|W_{i-1}|.
\end{equation*}
It follows that $W_i$ is non-empty as long as \[i-1\le-\log_{(1-d)(1-1/u)} \bigpar{|W|/u} =: I\bigpar{W|},\]
so we terminate with an independent set $\{w_1,\dots, w_j\} \subseteq W$ of size $j \ge \bigfloor{I(|W|)+1} \ge I(|W|)$.
\end{proof}
\begin{proof}[Proof of \refL{Degeneracy}]
We apply induction on the number of vertices.
The base case~${|V(G)|=1}$ is trivial. 
For the induction step~$|V(G)|>1$, we pick a vertex~$v$ of minimum degree, 
inductively color the subgraph~$G-v$ (obtained by removing~$v$) using~$r+1$ colors, 
and color~$v$ with a color not used by its at most~$r$~neighbors. 
\end{proof}

\section{Appendix: proof of \refL{isoledges}}\label{apx:isoledges}
\begin{proof}[Proof of \refL{isoledges}]
We keep the setup from the proof-outline in \refS{small}.
In particular, we define ${\bfG \subseteq \Gnq}$ as the induced subgraph of~$\Gnq$ consisting of all connected components of size at least three, 
define~$X_1$ as the number of isolated edges in~$\Gnq$, and define $X_2$ as the maximum size (counting number of edges) of a matching in~$\bfG$. 
Turning to the typical number $Y:=n-|V(\bfG)|$ of vertices in components of size at most two, 
note that the expected number of isolated vertices in~$\Gnq$ is
\[\lambda_0:=n (1-q)^{n-1} \sim n e^{-qn},\]
and that the expected number of isolated edges in~$\Gnq$ is
\[\lambda_1:=\binom{n}{2} q (1-q)^{2n-4} \sim \frac{n^2qe^{-2qn}}{2},\]
where we use the shorthand~$a_n \sim b_n$ for~$a_n = (1+o(1)) b_n$ to avoid clutter. 
Note that~$\E Y=\lambda_0+2\lambda_1$.
By a routine second moment calculation, for suitable~$\eps_n=o(1)$ it follows that with high probability 
\begin{equation}\label{eq:Y:whp}
Y=(1\pm\eps_n)(\lambda_0+2\lambda_1) .
\end{equation}
Deferring the choice of the sufficiently small constant~${d=d(\eps)>0}$, set
\begin{equation}\label{def:N}
N:=\floor{dn\sqrt{q} e^{-qn}} \sim d \sqrt{2\lambda_1} \gg 1 .
\end{equation}
Since~$M=X_1+X_2$, with~$X_2$ determined by~$\bfG$, 
for any interval~$I$ of length~$N$ it follows~that 
\begin{equation}\label{eq:conditioning}
\begin{split}
\bP(M\in I) & \; \le \; \Pr\bigpar{\text{estimate~\eqref{eq:Y:whp} fails}} + \sum_{G} \bP\bigpar{X_1+X_2\in I \mid \bfG=G} \bP(\bfG=G)\\ 
& \; \le \; o(1) \: + \: \max_{G,J} \bP\bigpar{X_1\in J \mid \bfG=G} , 
\end{split}
\end{equation}
where~$J$ is taken over all intervals of length~$N$, 
and~$G$ is taken over all graphs for which (i)~all connected components all have size at least three 
and (ii)~its number of vertices~$n-Y$ is compatible with~\eqref{eq:Y:whp}. 

To complete the proof, it suffices to show that the final probability appearing in~\eqref{eq:conditioning} 
is at most~$\eps$ for sufficiently large~$n$. 
To this end we henceforth fix a graph~$G$ that occurs in the maximum of~\eqref{eq:conditioning} described above, 
and let~$J$ be any interval of length~$N$ that maximizes~${\bP(X_1\in J \mid \bfG=G)}$. 
We define~$\ifG \subseteq \Gnq$ as the induced subgraph of~$\Gnq$ consisting of all connected components of size at most two. 
Note that~$\ifG$ has~$Y$ vertices, with~$Y=n-|V(\bfG)|$ determined by~$\bfG$. 
After conditioning on~$\bfG=G$, a moment's thought reveals that~$\ifG$ has the same distribution as~$G_{Y,q}$ conditioned on all connected components having size at most two.
Since there are~$\tfrac{1}{u!}\prod_{0 \le i < u}\binom{Y-2i}{2}$ graphs on~$Y$ vertices that consist of~$u$~isolated edges and~${Y-2u}$~isolated vertices, 
it follows for any integer~${0 \le m \le Y/2-1}$~that 
\begin{equation}\label{ratio}
\begin{split}
\Phi_G(m) &:= \frac{\bP\bigpar{X_1=m \mid \bfG=G}}{\bP\bigpar{X_1=m+1 \mid \bfG=G}} \\
& = \frac{\tfrac{1}{m!}\prod_{0 \le i < m}\binom{Y-2i}{2} \cdot q^{m} (1-q)^{\binom{Y}{2}-m}}{\tfrac{1}{(m+1)!}\prod_{0 \le i \le m}\binom{Y-2i}{2} \cdot q^{m+1} (1-q)^{\binom{Y}{2}-(m+1)}} = \frac{(1-q)(m+1)}{q\binom{Y-2m}{2}}.
\end{split}
\end{equation}
To get a handle on these probabilities, first note that~$\lambda_0 = \Omega(\lambda_1)$ follows from $\lambda_1/\lambda_0 = \Theta(nqe^{-nq}) = O(1)$. 
Using~\eqref{eq:Y:whp} we then deduce $\sqrt{q} \: Y = \Theta(\sqrt{q}\lambda_0) = \Theta(n\sqrt{q} e^{-qn}) \gg 1$ as well as 
\[ 
\Phi_G\bigl(0\bigr) 
= \frac{\Theta(1)}{(\sqrt{q} \:Y)^2} \ll 1
\qquad\text{ and }\qquad 
\Phi_G\bigl(Y/2-1\bigr) 
= \frac{\Theta(\sqrt{q}\:Y)}{q^{3/2}} \gg 1 , 
\]
so by the intermediate value theorem there exists a real~${0 < m_0< Y/2-1}$ satisfying ${\Phi_G(m_0)=1}$.
Noting that~$\Phi_G(m)$ is a strictly increasing function, it now is routine to check that ${m_0\sim \lambda_1}$, 
since for any ${m \sim \lambda_1}$ we have ${Y-2m \sim \lambda_0 + o(\lambda_1) \sim \lambda_0}$ and~${\Phi_G(m) \sim \lambda_1/(q \lambda_0^2/2) \sim 1}$. 
Using~$\Phi_G(m_0)=1$ together with ${Y-2m_0 \sim \lambda_0 = \Omega(\lambda_1)}$ and ${m_0 \sim \lambda_1}$, 
for any integer~$m=m_0+i$ with $i=o(\lambda_1)$ we similarly infer~that 
\begin{equation}\label{controlratio}
\begin{split}
\Phi_G(m) & = \frac{\Phi_G(m_0+i)}{\Phi_G(m_0)} = \frac{\binom{y-2m_0}{2}}{\binom{y-2m_0-2i}{2}} \frac{m_0+i+1}{m_0+1} \\
& = \left(1+\frac{2i}{y-2m_0-2i+O(1)}\right)^2 \left(1+\frac{i}{m_0+1}\right) =1+\frac{\Theta(i)}{\lambda_1}.
\end{split}
\end{equation}
Since the probability ratio~$\Phi_G(m)$ from~\eqref{ratio} is a strictly increasing function, 
the interval~$J$ must intersect~${\bigl\{\ceil{m_0}-1,\ldots, \floor{m_0}+1\bigr\}}$ in at least one element, 
because otherwise we could increase the probability ${\bP(X_1\in J \mid \bfG=G)}$ 
by shifting the interval~$J$ by plus or minus one (contradicting maximality). 
Furthermore, by definition of~$m_0$, we have~${\Phi_G(m) \le 1}$ for all~$m \le m_0$. 
By choosing the constant~$d=d(\eps)>0$ sufficiently small, 
for all integers~$m \in J$ and~$1 \le s\le \ceil{2/\eps}$ 
it follows in view of~\eqref{ratio}, \eqref{controlratio} and~\eqref{def:N}~that
\begin{equation}\label{ratio:m}
\begin{split}
\frac{\bP\bigpar{X_1=m \mid \bfG=G}}{\bP\bigpar{X_1=m+2Ns \mid \bfG=G}} 
= \prod_{0\le i<2sN} \Phi_G(m) 
\le \left(1+\frac{\Theta(sN)}{\lambda_1}\right)^{2Ns} \le e^{O(s^2d^2)} \le 2 . 
\end{split}
\end{equation}
Since the interval~$J$ has length~$N$, 
by applying~\eqref{ratio:m} to each~$m \in J$ it then follows that 
\begin{align*}
2/\eps \cdot \bP\bigpar{X_1\in J \mid \bfG=G} 
\le  \sum_{1 \le s \le \ceil{2/\eps}} 2 \; \bP\bigpar{X_1-2Ns \in J \mid \bfG=G} \le 2 ,
\end{align*}
which together with estimate~\eqref{eq:conditioning} completes the proof of \refL{isoledges}, as discussed. 
\end{proof}

\end{appendix}

\end{document}